\documentclass{amsart}

\usepackage[dvipsnames]{xcolor}
\usepackage[english]{babel}
\usepackage{amsfonts}
\usepackage{amsthm}
\usepackage{amsbsy}
\usepackage{amssymb}
\usepackage{graphicx}
\usepackage{eso-pic}
\usepackage{tikz}
\usepackage{xfrac}
\usepackage{float}
\usepackage[bookmarks=true,colorlinks=true,urlcolor=black,linkcolor=black,citecolor=black,pagebackref=false]
{hyperref}  
\usepackage{bookmark}
\usepackage{amsmath}
\usepackage{subfigure}
\usepackage{hyperref}
\usepackage{scalerel,stackengine}
\usepackage{epigraph}
\usepackage{placeins}
\usetikzlibrary{matrix,arrows}
\usepackage{enumitem}
\usepackage{microtype}
\usepackage{mathtools}
\usepackage{hhline}

\usepackage[doi=false,isbn=false,
maxbibnames=99,backend=biber,style=alphabetic,sorting=nyt,url=false]{biblatex}
\usepackage{csquotes}
\usepackage{amsmath,amssymb,tikz-cd}
\usepackage[inline]{asymptote}
\usepackage{comment}
\usepackage{parskip}

\makeatletter
\newcommand{\mylabel}[2]{#2\def\@currentlabel{#2}\label{#1}}
\makeatother
\usepackage{todonotes}

\newtheorem{lemma}{Lemma}[section]
\newtheorem{thm}[lemma]{Theorem}

\newtheorem*{thmx*}{Theorem~6.1}
\newtheorem*{thmy*}{Theorem~7.1}
\newtheorem{prop}[lemma]{Proposition}
\newtheorem{cor}[lemma]{Corollary}
\newtheorem*{cor*}{Corollary}

\theoremstyle{definition}
\newtheorem{defn}[lemma]{Definition}

\newtheorem{quest}[lemma]{Question}

\theoremstyle{remark}
\newtheorem{rem}[lemma]{Remark}

\newcommand\restr[2]{{
  \left.\kern-\nulldelimiterspace 
  #1 
  \littletaller 
  \right|_{#2} 
  }}

\newcommand{\littletaller}{\mathchoice{\vphantom{\big|}}{}{}{}}

\newcommand{\fac}[1]{\bf{#1}}

\newcommand{\matR} {\ensuremath {\mathbb{R}}}
\newcommand{\matQ} {\ensuremath {\mathbb{Q}}}
\newcommand{\matZ} {\ensuremath {\mathbb{Z}}}

\newcommand{\matH} {\ensuremath {\mathbb{H}}}

\newcommand{\Isom}{\ensuremath {\mathrm{Isom}}}

\author{Jacopo Guoyi Chen \and Edoardo Rizzi}
\address{Scuola Normale Superiore \newline Piazza dei Cavalieri 7, 56126 Pisa, Italy}
\email{jacopo.chen@sns.it, edoardo.rizzi@sns.it}

\thanks{}

\title{Cusp-transitive 4-manifolds with every cusp section}

\begin{document}

\begin{abstract}
We realize every closed flat 3-manifold as a cusp section of a complete, finite-volume hyperbolic 4-manifold whose symmetry group acts transitively on the set of cusps. Moreover, for every such 3-manifold, a dense subset of its flat metrics can be realized as cusp sections of a cusp-transitive 4-manifold. Finally, we prove that there are a lot of 4-manifolds with pairwise isometric cusps, for any given cusp type.
\end{abstract}

\maketitle

\section{Introduction}

This work is a continuation of~\cite{rizzi}, where the second-named author constructed some new cusp-transitive hyperbolic $4$-manifolds. We recall in the following the general settings and results.

We say that a complete, finite-volume hyperbolic manifold is \emph{cusp transitive} if the action of its isometry group is transitive on its cusps.
This condition is trivially satisfied by manifolds with a single cusp (also called \emph{$1$-cusped} manifolds); examples are known in dimension $2$, $3$ (infinitely many with bounded volume, in fact) and $4$~\cite{KM}.
For $n>4$, the existence of $1$-cusped hyperbolic $n$-manifolds is unknown, and in general there are no explicit examples of cusp-transitive manifolds in the literature. However, it is easy to obtain examples of the latter in dimension up to $8$ by coloring some well-known right-angled polytopes.

The \emph{type} of a cusp of a hyperbolic manifold is the diffeomorphism class of its section, which is a closed flat hypersurface.
In this article we will look at the $4$-dimensional case; the possible cusp types are closed flat $3$-manifolds. There are $10$ such manifolds~\cite{conway-rossetti}; six of them are orientable:
\begin{itemize}
    \item $E_1$, the $3$-torus;
    \item $E_2$, the $\frac 12$-twist manifold;
    \item $E_3$, the $\frac 13$-twist manifold;
    \item $E_4$, the $\frac 14$-twist manifold;
    \item $E_5$, the $\frac 16$-twist manifold;
    \item $E_6$, the Hantzsche--Wendt manifold;
\end{itemize}
while four are non-orientable: we will call them $B_1$, $B_2$, $B_3$ and $B_4$ (respectively denoted by $+a1$, $-a1$, $+a2$ and $-a2$ in~\cite{conway-rossetti}). These $10$ manifolds can be constructed as in 
Figures~\ref{fig:N} and~\ref{fig:N2} by gluing facets of polyhedral fundamental domains; the gluing maps can be obtained by inspecting~\cite[Table~12]{conway-rossetti}.

In the case of $1$-cusped orientable $4$-manifolds, the cusp types $E_1$~\cite{KM} and $E_2$~\cite{KS} can be realized, while $E_3$ and $E_5$ cannot~\cite{LR}.
More generally, the types $E_6$~\cite{FKS} and $E_4$~\cite{rizzi} can be realized by cusp-transitive $4$-manifolds. The construction in the latter paper also realizes $E_1$, $E_2$ and $E_6$ as cusp types. As for the non-orientable case, there exist $1$-cusped $4$-manifolds with cusp type $B_1$~\cite{B1-mfd} and $B_2$~\cite{B2-mfd}.

We improve on the technique of~\cite{rizzi}, by requiring a less explicit construction, and we prove that the remaining cusp types can also be realized; indeed, our construction actually realizes all the cusp types. Specifically we show:

\begin{thm}\label{thm:main0}
For each closed flat $3$-manifold $N$ there exists a cusp-transitive hyperbolic $4$-manifold $M$ with cusps of type $N$.
\end{thm}
By Remark~\ref{rem:orient}, if $N$ is orientable, we can also take $M$ to be orientable.
The proof of this theorem is split into two parts: in Section~\ref{sec:layered} we discuss the case of the manifolds $E_1$, $E_2$, $E_4$, $E_6$, $B_1$, $B_2$, $B_3$ and $B_4$, while Section~\ref{sec:remaining} is dedicated to $E_3$ and $E_5$. The $4$-manifolds constructed in each section are all commensurable to each other, but they are arithmetic in the first case and non-arithmetic in the second. Hence, they fall into two commensurability classes.

Using an argument inspired by~\cite{N}, this result is strengthened in Section~\ref{sec:density} as follows:
\begin{thmx*}
For every closed flat $3$-manifold $N$, the set of flat metrics on $N$ which can be realized as cusp sections of a cusp-transitive $4$-manifold is dense in the space of all flat metrics of $N$.
\end{thmx*}
Moreover, with the same technique, we show an analogous result in dimension $3$ (where the possible cusp types are the torus and the Klein bottle).

Finally, in Section~\ref{sec:lots}, a variant of our method, combined with some arguments of \cite{MR1952926,MR2726109}, enables the construction of a lot of manifolds with pairwise isometric cusps:

\begin{thmy*}
    For every closed flat $3$-manifold $N$, there exists a positive constant $c$ such that, for sufficiently large $V>0$, there exist at least $V^{cV}$ complete hyperbolic $4$-manifolds with pairwise isometric cusps of type $N$ and volume $\le V$. 
\end{thmy*}

The construction is very similar to the one of~\cite{rizzi}. Our manifolds are built by applying Davis' \emph{basic construction}~\cite{D} to a so-called \emph{developable reflectofold}, in turn obtained by gluing some copies of a hyperbolic Coxeter polytope with the following two properties:

\begin{itemize}
\item[\mylabel{item:a}{(a)}] it has exactly one ideal vertex;
\item[\mylabel{item:b}{(b)}] 
if a bounded facet and an unbounded facet intersect, then 
their dihedral angle is 
an even submultiple of $\pi$.
\end{itemize}




We have found two hyperbolic Coxeter $n$-polytopes with $n \geq 4$ satisfying \ref{item:a} and \ref{item:b}: a $4$-polytope $P$ among Im Hof's polytopes associated to Napier cycles
\cite{IH, felikson}, and another $4$-polytope $V$ with $8$ facets, which as far as we know does not appear in the literature. Hence, we have not been able to build manifolds of dimension greater than $4$ using these techniques.

The polytope $V$ allows us to answer Questions~1.2 and 1.3 from~\cite{rizzi} (reported here) in the affirmative:

\begin{itemize}
    \item Does there exist a finite-volume hyperbolic Coxeter polytope of dimension $n \geq 4$ satisfying \ref{item:a} and \ref{item:b}? If the dimension is $n=4$ we require that the link of the ideal vertex is neither a parallelepiped nor a prism over a $(2,4,4)$-triangle.
    \item Does there exist a cusp-transitive hyperbolic $4$-manifold with cusps of type $E_3$ or $E_5$?
\end{itemize}
Indeed, the link of the ideal vertex of $V$ is a prism $K$ over an equilateral triangle; the fact that $E_3$ and $E_5$ are tessellated by copies of $K$ is crucial to the construction of the corresponding cusp-transitive manifolds. The answer to the first question naturally leads to the following.

\begin{quest}
Does there exist a finite-volume hyperbolic Coxeter polytope of dimension $n \ge 5$ satisfying \ref{item:a} and \ref{item:b}?
\end{quest}

The paper is organized as follows. In Section~\ref{sec:reflect} we describe how to obtain a cusp-transitive manifold from a $1$-cusped developable reflectofold. In Section~\ref{sec:proof} we introduce the polytope $P$ and assemble $16$ copies of it to create a bigger polytope $Q$. In Section~\ref{sec:layered} we show how to construct reflectofolds by gluing copies of $Q$ according to some cubical tessellations, and consequently prove the first case of Theorem~\ref{thm:main0}. In Section~\ref{sec:remaining} we introduce the polytope $V$ and construct reflectofolds by gluing copies of it according to certain tessellations in prisms, proving the second and final case of Theorem~\ref{thm:main0}. In Section~\ref{sec:density}, we prove Theorem~\ref{thm:density} and its $3$-dimensional counterpart. Lastly, in Section~\ref{sec:lots}, we prove Theorem~\ref{thm:lots}.

We would like to thank our advisors Bruno Martelli and Stefano Riolo for their support.

%
%
%

\section{Reflectofolds}\label{sec:reflect}
Let $N$ be a closed flat $3$-manifold. In this section we will see that, in order to show the existence of a cusp-transitive manifold with cusp type $N$, it suffices to prove that there exists a $1$-cusped \emph{developable reflectofold} with cusp type $N$. We give some definitions in the following (compare~\cite{davis-lectures, rizzi}).

\begin{defn}
    We say that a hyperbolic manifold $R$ with cornered boundary is a \emph{reflectofold} if $R$ is locally a hyperbolic Coxeter polytope.   
\end{defn}

Let $R$ be a reflectofold. Since each local model $P$ of $R$ is stratified into $k$-dimensional faces, for $k = 0, \ldots, n$, there is a naturally induced stratification of $R$ into maximal, connected, totally geodesic submanifolds with boundary, called $k$-\emph{faces}. We will call \emph{facets} and \emph{corners} the $(n-1)$-faces and $(n-2)$-faces of $R$, respectively. We can define the \emph{dihedral angle} of a corner as the dihedral angle of the corresponding ridge of a local model.


\begin{defn}
   A reflectofold is \emph{developable} if it has the following properties:
    \begin{itemize}
        \item[\mylabel{EF}{(EF)}] \emph{Embedded faces}: Each corner is contained in the intersection of two distinct facets.
        \item[\mylabel{AC}{(AC)}] \emph{Angle consistency}: All the corners in the intersection of two distinct facets have equal dihedral angles.
    \end{itemize}    
\end{defn}

\begin{prop}\label{prop:rizzi}
    Let $R$ be a $1$-cusped developable reflectofold with complete cusp section, and built by gluing finitely many finite-volume polytopes. Then, there exists a cusp-transitive manifold $M$ having cusps isometric to that of $R$.
\end{prop}
\begin{proof}
    The desired manifold $M$ can be constructed as in~\cite[Section~2]{rizzi} by gluing finitely many copies of $R$ following Davis' basic construction~\cite{D}. Since $R$ has finite volume, $M$ will also have finite volume.
    Note that, unlike~\cite{rizzi}, we do not require reflectofolds to be complete; however, $M$ will be complete since it is obtained by gluing polytopes, and has complete cusp sections (see~\cite[Theorem~11.1.6]{complete}).
\end{proof}

\begin{rem}\label{rem:orient}
    If the cusp section of the cusp-transitive manifold $M$ thus constructed is orientable, we can also construct an orientable cusp-transitive manifold with the same cusp section, namely the orientable double cover of $M$, as in~\cite[Corollary~2.4]{rizzi}.
\end{rem}

\section{Some polytopes} \label{sec:proof}

In this section we will introduce and construct the polytopes that we are going to use in Section~\ref{sec:layered}.

Following Vinberg's paper~\cite{V}, we define a \emph{hyperbolic Coxeter polytope} as a convex polytope $P \subset \matH^n$ with finitely many facets, with dihedral angles of the form $\pi/m$ for an integer $m\ge 2$. We call \emph{facets} and \emph{ridges} the faces of $P$ of codimension $1$ and $2$, respectively.

We associate to such a polytope its \emph{Coxeter diagram}, a labeled graph defined as follows. We refer to the facets by progressive numbers $1,\dots,k$. The graph has a vertex for each facet, and if two facets $i$ and $j$ intersect with dihedral angle $\pi/m_{ij}$, then the edge joining the corresponding vertices has label $m_{ij}$. If $m_{ij} = 2$, the edge is conventionally omitted.
If the supporting hyperplanes of $i$ and $j$ are tangent at infinity, we label the corresponding edge with $m_{ij} = \infty$.
Finally, a dashed edge is drawn between any two vertices corresponding to ultraparallel facets $i$ and $j$. The edge can be labeled by $\cosh(d_{ij})$, where $d_{ij}$ is the distance between the supporting hyperplanes.

The \emph{Gram matrix} $G$ of the polytope is a real symmetric $k\times k$ matrix, defined as follows:
\begin{equation}
    G_{ij} \coloneqq 
    \begin{cases}
        1 & \text{if $i = j$;} \\
        -\cos(\pi/m_{ij}) & \text{if the facets $i$ and $j$ intersect;} \\
        -1 & \text{if the facets $i$ and $j$ are parallel;} \\
        -\cosh(d_{ij}) & \text{if the facets $i$ and $j$ are ultraparallel.}
    \end{cases}
\end{equation}
\subsection{The polytope \texorpdfstring{$P$}{P}}\label{sec:21}
We now introduce a Coxeter polytope from~\cite{IH}.

Consider the following Coxeter diagram $D_1$:

\begin{center}
\begin{tikzpicture}[thick,scale=0.8, every node/.style={transform shape}]
\node[draw, circle, inner sep=2pt, minimum size=3mm] (1) at (-1.56,-1.25)  {6};
\node[draw, circle, inner sep=2pt, minimum size=3mm] (2) at (0,-2)         {5};
\node[draw, circle, inner sep=2pt, minimum size=3mm] (3) at (1.56,-1.25)   {4};
\node[draw, circle, inner sep=2pt, minimum size=3mm] (4) at (1.95,0.45)  {3};
\node[draw, circle, inner sep=2pt, minimum size=3mm] (5) at (0.87,1.80)  {2};
\node[draw, circle, inner sep=2pt, minimum size=3mm] (6) at (-0.87,1.80) {1};
\node[draw, circle, inner sep=2pt, minimum size=3mm] (7) at (-1.95,0.45) {7};

\node at (-0.87,-1.80)  {$4$};
\node at (0.87,-1.80)   {$4$};
\node at (1.95,-0.45)   {$6$};
\node at (1.56,1.25)  {$4$};
\node at (0,2)        {$\infty$};
\node at (-1.56,1.25) {$4$};
\node at (-1.95,-0.45)  {$6$};

\draw (1) -- (2) node[above,midway]   {};
\draw (2) -- (3) node[above,near end] {};
\draw (3) -- (4) node[above,midway]   {};
\draw (4) -- (5) node[above,midway]   {};
\draw (5) -- (6) node[above,midway]   {};
\draw (6) -- (7) node[above,midway]   {};
\draw (7) -- (1) node[above,midway]   {};
\end{tikzpicture}
\end{center}

This graph is the Coxeter diagram of a finite-volume arithmetic hyperbolic Coxeter $4$-polytope $P$, introduced in \cite{IH}, which satisfies \ref{item:a} and \ref{item:b}. Indeed, by \cite{V}, the ideal vertices correspond to the maximal affine subdiagrams of the Coxeter diagram, and in $D_1$ we have exactly one of this kind (see \cite[Table 2]{V}), spanned by the vertices $1,2,4,5,6$. The subdiagram spanned by these vertices is the Coxeter diagram of a Euclidean right prism over a triangle with interior angles $\pi/2$, $\pi/4$, $\pi/4$, which is the horospherical link of the ideal vertex of $P$. Arithmeticity can be verified 
using the program CoxIter~\cite{coxiter,coxiter2}.


\begin{figure}
    \centering
\includegraphics{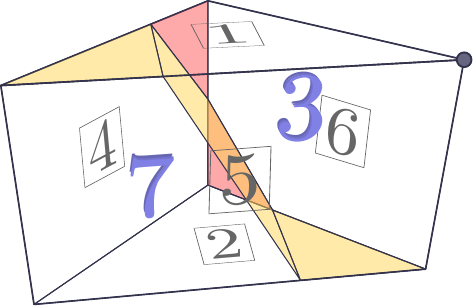}
    \color{black}
    \caption{The projection of $P$ onto the link $L$ of its ideal vertex. The five faces of $L$ and the two interior regions are labeled with the corresponding facets of $P$: five non-compact ($\fac{1},\fac{2},\fac{4},\fac{5},\fac{6}$) and two compact ($\fac{3},\fac{7}$). 
    }
    \label{fig:prismL}
\end{figure}

\subsection{Visualizing polytopes with one ideal vertex}\label{sec:visualizing}
In order to better understand the geometry of the $4$-dimensional polytope we just defined, it may prove useful to \emph{project} the compact boundary facets onto the horospherical link of the ideal vertex as follows.

Place the polytope $P$ in the half-space model of $\matH^4$ with the ideal vertex at infinity, so that the non-compact facets of $P$ become vertical, and let $\pi: \matH^4 \to \matR^3$ be the orthogonal projection onto a horizontal hyperplane, such as a horosphere or the ideal boundary. The image of $P$ under $\pi$ is the three-dimensional link $L$ of the ideal vertex, and the non-compact facets are mapped to its boundary.

We can compute the face lattice of $P$ by enumerating all spherical subdiagrams of $D$, and therefore determine the shape of the two regions associated to the facets $\fac{3}$ and $\fac{7}$ (Figure~\ref{fig:prismL}) by applying the following result.
\begin{prop}\label{prop:comb-iso}
Let $Z$ be an $n$-polytope in $\matH^n$ with one ideal vertex $v$. Let $\Lambda$ be the link of $v$. The images of the compact facets of $Z$ under the projection onto a horosphere centered at $v$ partition $\Lambda$ into a polyhedral complex, whose face lattice is combinatorially isomorphic to the lattice of compact faces of $Z$.
\end{prop}
\begin{proof}
    We consider the half-space model where $v$ is at infinity. We say that a $k$-subspace of the half-space model is \emph{vertical} if its ideal boundary contains $v$.

    Let $F$ be a compact $k$-face of $Z$. Then, in the half-space model, $F$ is supported on a $k$-hemisphere $H$. Let $\Sigma$ be the unique vertical $(k+1)$-space containing $H$. The face $F$ is bounded by the supporting $(k-1)$-spaces of its facets $F_1, \dots, F_m$. Each $F_i$ lies on a unique vertical $k$-space $S_i$. As such, $F$ is the intersection of $H$ and some hyperbolic half-$(k+1)$-spaces bounded by the $S_i$ and contained in $\Sigma$. Let $P_F$ be the Euclidean $k$-polytope obtained by intersecting the projections of these half-spaces on the horosphere; then we have $F = (P_F \times (0,+\infty))\cap H$. Hence, the projection is a homeomorphism between $F$ and $P_F$, which preserves facets. Since $Z$ has only one ideal vertex, every compact face of $Z$ is contained in a compact facet. Hence, by induction on the codimension of faces, the projection preserves the face lattice of the compact boundary of $Z$.
\end{proof}

Note that dihedral angles between facets of $P$ are defined along ridges (codimension-$2$ faces), which appear in Figure~\ref{fig:prismL} either as $2$-faces (if the angle involves a compact facet) or as edges of $L$ (if the angle is between two non-compact facets). 
We mark the former by coloring certain $2$-faces of the projection, with the following rule. Every $2$-face corresponds to a ridge between either two compact facets or a compact facet and a non-compact one. In the first case, we assign the dihedral angle to the $2$-face, while in the second case, we assign the doubled dihedral angle. We use yellow for $\pi/2$, red for $\pi/3$, and we leave the $2$-face transparent for $\pi$.

This rule keeps track of the fact that, if we glue two copies of $P$ along a non-compact facet $F$, the angle along any ridge it shares with a compact facet gets doubled.
For instance, the dihedral angle between facets $\fac{1}$ and $\fac{7}$ is $\pi/4$; when doubling $P$ along $\fac{1}$, the angle becomes $\pi/2$, so we color in yellow the corresponding triangle in Figure~\ref{fig:prismL}.

Moreover, after doubling, some ridges end up with an angle of $\pi$, meaning that the compact facet is coplanar with its mirrored copy. Leaving the $2$-face transparent has the advantage of visually representing when compact facets merge together in an arbitrary gluing of copies of $P$.

\begin{figure}
    \centering
\setlength{\fboxsep}{1pt}
\setlength{\fboxrule}{0.1pt}
\begin{tabular}{cc}
    \includegraphics{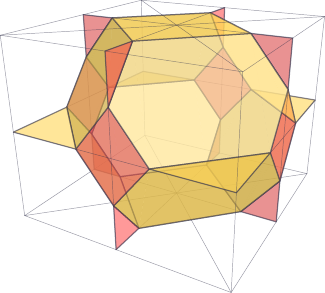}
 &  \includegraphics{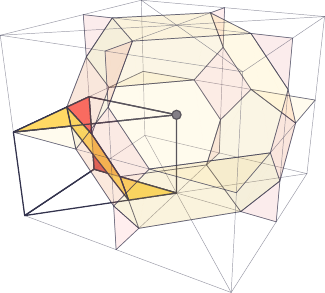}
\end{tabular}
    \color{black}
    \caption{On the left, the projection of $Q$ onto the link $C$. The truncated octahedron corresponds to a facet made of $16$ copies of the facet $\fac{3}$ of $P$.
    The other $8$ regions appear when copies of the facet $\fac{7}$ merge together two at a time. 
    On the right, we emphasize one copy of $L$ inside $C$.
    }
    \label{fig:cuboid}
\end{figure}

\subsection{The construction of the polytope \texorpdfstring{$Q$}{Q}}
In this section we will take $16$ copies of $P$ and we will glue them in order to form a bigger polytope $Q$.

Consider the subdiagram of $D$ spanned by the vertices $1,5,6$. This induces a subgroup $G \simeq \matZ_2 \times D_4$, of cardinality $16$, which is the stabilizer of a non-compact edge of $P$. This edge projects to the marked vertex in Figure~\ref{fig:prismL}.

We can define a new polytope $Q$ by taking the orbit of $P$ under $G$. Equivalently, we glue $16$ copies of $P$ along their non-compact facets $\fac{1},\fac{5},\fac{6}$ using the identity map. By construction, $Q$ has exactly one ideal vertex, whose link is a right prism $C$ with a square base, obtained by placing $16$ copies of $L$ around the marked vertex (Figure~\ref{fig:cuboid}).

The facets $\fac{3}$ merge together into a single facet of $Q$ (which we will also call $\fac{3}$), which projects to an irregular truncated octahedron in $C$, with $8$ hexagonal and $6$ quadrilateral facets.
The latter correspond to ridges of $Q$ where the facet $\fac{3}$ meets the non-compact facets: the dihedral angles are $\pi/6$ for the vertical quadrilaterals and $\pi/4$ for the horizontal ones.

The height, width and depth of $C$ are in a ratio of $\cos(\pi/4) : \cos(\pi/6) : \cos(\pi/6) = \sqrt{2} : \sqrt{3} : \sqrt{3}$.
Indeed, suppose that $C$ is the cuboid $[-x_1,x_1] \times [-x_2,x_2] \times [-x_3,x_3]$. Then, in the conformal half-space model, each non-compact facet of $Q$ is contained in the product of a facet of $C$ and $(0,+\infty)$. 
The facet $\fac{3}$ is supported on a hemisphere, which is centered at $(0,0,0,0)$ by symmetry. It is not hard to see that, if this hemisphere has radius $r$, then the acute angles with the supporting hyperplanes of the non-compact facets of $Q$ are $\theta_i \coloneqq \arccos(x_i/r)$. Hence, the cosines of the $\theta_i$ are in the same ratios as the $x_i$.

Because of this, all symmetries of $C$ must preserve or exchange the two horizontal faces. 
There are $16$ such symmetries, and they are generated by reflections in the facets $\fac{1},\fac{5},\fac{6}$ of $L$, so they extend to symmetries of $Q$ and they form a group isomorphic to $G$.
This group can also be defined \emph{a priori} as the group of symmetries of a \emph{combinatorial} cube preserving or exchanging a certain pair of opposite facets. 

\begin{rem}
The polytope $C$ naturally tessellates $\matR^3$ by translations. This gives a tessellation into truncated octahedra in the following way: some are centered and contained in the copies of $C$, as in Figure~\ref{fig:cuboid}, while the others are centered at the vertices of the tessellation; one-eighth of a truncated octahedron can be seen near each vertex of $C$ in Figure~\ref{fig:cuboid}. The two types of truncated octahedra are actually congruent, because of the symmetry of $D$ that exchanges $3$ and $7$. Indeed, passing to the dual tessellation by copies of $C$ exchanges the two types of truncated octahedra.
\end{rem}

\section{Layered tessellations of 3-manifolds}\label{sec:layered}
In this section we will prove that if a flat $3$-manifold $N$ admits a tessellation in cubes with some properties, then there is a cusp-transitive hyperbolic $4$-manifold with cusp type $N$.

\begin{defn}[Layered tessellation]
Let $T$ be a tessellation of a flat $3$-manifold $N$ into Euclidean cubes, 
with some chosen \emph{special} $2$-cells.
We say that $T$ is \emph{layered} if each cube has two opposite special facets, and whenever two cubes $C_1$, $C_2$ share a facet $F$, the reflection through $F$ sends the special facets of $C_1$ to those of $C_2$, and vice versa. 
\end{defn}

The above condition causes the special facets to fall into several embedded geodesic surfaces tessellated by squares, in a discrete analog of a foliation.

\begin{rem}\label{rem:comb-lay}
We will only make use of the combinatorial properties of layered tessellations; hence, we may define generalizations, such as layered tessellations by rectangular cuboids, in the same way.    
\end{rem}

\begin{figure}
    \centering
\begin{tabular}{rcl}
\includegraphics{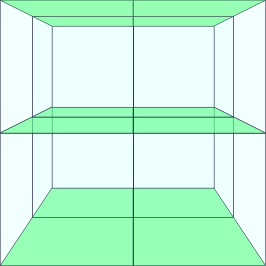}
& \raisebox{14.25ex}{\Large $\longrightarrow$} & 
\includegraphics{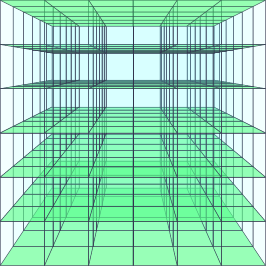}
\end{tabular}
    \color{black}
    \caption{Subdividing a layered tessellation. Special faces are colored green.}
    \label{fig:subdiv}
\end{figure}

\begin{rem}\label{rem:subdiv}
    A layered tessellation can be \emph{subdivided} by replacing every cube with a block of $n\times n\times n$ cubes, $n \ge 1$, in which we mark as special all facets parallel to the two original special facets (see Figure~\ref{fig:subdiv}).
\end{rem}

\begin{defn}
    We say that a layered tessellation is \emph{proper} if the following dual conditions hold:
    \begin{itemize}
        \item every cube is distinct from its six neighbors, which are pairwise distinct;
        \item every vertex is distinct from its six neighbors, which are pairwise distinct.
    \end{itemize}
\end{defn}

\begin{lemma}\label{lemma:proper}
    Every layered tessellation can be subdivided into a proper one.
\end{lemma}
\begin{proof}
    First, we realize each cube of the tessellation as a cube with side length $1$. This gives a flat Riemannian metric on a compact $3$-manifold, which has an injectivity radius $r>0$. We then subdivide the tessellation as in Remark~\ref{rem:subdiv}, in such a way that the side length of the little cubes is less than $r$. The resulting tessellation is proper because, for every cube, the centers of it and its neighbors are contained in an embedded ball, and a similar argument applies to the vertices.  
\end{proof}

\begin{thm}\label{thm:main} 
    Let $N$ be a closed $3$-manifold that admits a layered tessellation. Then there exists an arithmetic cusp-transitive hyperbolic $4$-manifold with cusp type $N$.
\end{thm}
\begin{proof} 
    By Proposition~\ref{prop:rizzi}, it suffices to prove that there exists a $1$-cusped developable reflectofold with cusp type $N$, obtained by gluing copies of $Q$.

    By Lemma~\ref{lemma:proper}, we may assume that the tessellation is proper.
    
    We have that $N$ is obtained by gluing some copies of $C$ along their facets, where the gluing maps are restrictions of isometries of $G$: this gluing is induced by the layered tessellation,
    where the special facets correspond to the horizontal facets of $C$. 
    Since there is a natural correspondence between facets of $C$ and non-compact facets of $Q$, and every isometry in $G$ extends to $Q$ accordingly, we can glue copies of $Q$ in the same pattern. 
    This gives a 
    $1$-cusped reflectofold with cusp type $N$. 

    It remains to prove that $N$ is developable. The facets of $N$ are hyperbolic truncated octahedra, which arise from the merging of $16$ facets $\fac{3}$ or $\fac{7}$; we will call them \emph{of types} $\fac{3}$ and $\fac{7}$ respectively. The former are centered at the centers of the cubes, while the latter are centered at the vertices of the tessellation.

    If a facet of type $\fac{3}$ were adjacent to itself, it would be so along a quadrilateral corner, and so it would also occupy a neighboring cube. However, the two cubes must be distinct by properness of the tessellation. A similar argument involving vertices works for facets of type $\fac{7}$.

    As for angle consistency, if two facets of different type intersect, they do so in a hexagonal corner with a dihedral angle of $\pi/2$.
    If a facet intersects another of the same type, say $\fac{3}$, it must do so at a single quadrilateral corner, since the neighbors of a given cube are pairwise distinct by properness; angle consistency follows trivially.
    A dual argument deals with the case of two facets of type $\fac{7}$.

    The $4$-manifold thus constructed is arithmetic since it covers the arithmetic orbifold $P$.
\end{proof}

\begin{cor}\label{cor:eight-cusps}
    Let $N$ be one of the manifolds $E_1, E_2, E_4, E_6, B_1, B_2, B_3, B_4$. Then there exists a cusp-transitive arithmetic hyperbolic $4$-manifold with cusp type $N$.
\end{cor}
\begin{proof}
    As shown in Figure~\ref{fig:N}, the manifold $N$ admits a layered tessellation, so the result follows from Theorem~\ref{thm:main}.
\end{proof}

\begin{figure}[p]
    \centering
    \addtolength{\tabcolsep}{2em}
    \begin{tabular}{c|c}
\includegraphics{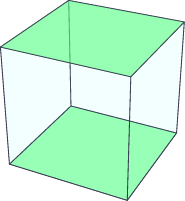}
& 
\includegraphics{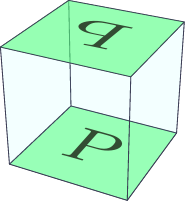}
\\[0.0ex]
\color{black} $3$-torus & $B_1$ \\[3ex]
\includegraphics{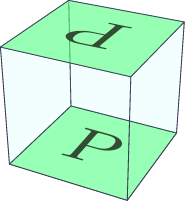}
&
\includegraphics{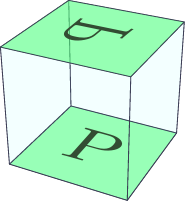}
\\[0.0ex]
\color{black} $\frac 12$-twist & $B_2$ \\[3ex]
\includegraphics{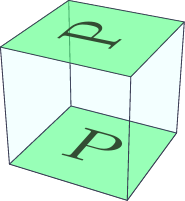}
&
\includegraphics{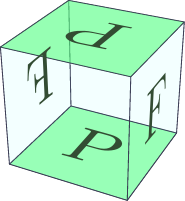}
\\[0.0ex]
\color{black} $\frac 14$-twist & $B_3$ \\[3ex]
\includegraphics{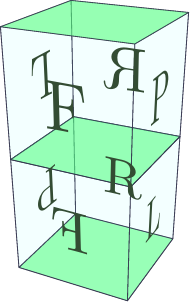}
&
\includegraphics{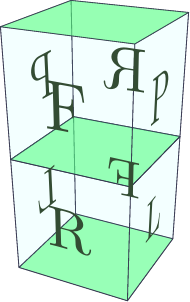}
\\[0.0ex]
\color{black} Hantzsche--Wendt & $B_4$         
    \end{tabular}
    \addtolength{\tabcolsep}{-2em}
    \color{black}
    \caption{Layered tessellations for eight closed flat $3$-manifolds. Labeled facets are glued according to their labels, while unlabeled facets are glued to the opposite facets by translation.}
    \label{fig:N}
\end{figure}

\section{The remaining manifolds}\label{sec:remaining}
In this section we will construct cusp-transitive manifolds having the $\frac 13$-twist and $\frac 16$-twist manifolds as cusp sections.

\subsection{The polytope \texorpdfstring{$V$}{V}}\label{sec:poly-V}
Let $D_2$ be the following Coxeter diagram:
\begin{center}
\begin{tikzpicture}[thick,scale=1.1,  every node/.style={scale=0.8}]
\node[draw, circle, inner sep=2pt, minimum size=3mm] (8) at (2,-1)  {8};
\node[draw, circle, inner sep=2pt, minimum size=3mm] (7) at (2,1)  {7};
\node[draw, circle, inner sep=2pt, minimum size=3mm] (6) at (1,0)   {6};
\node[draw, circle, inner sep=2pt, minimum size=3mm] (5) at (0,-1)  {5};
\node[draw, circle, inner sep=2pt, minimum size=3mm] (4) at (0,1)  {4};
\node[draw, circle, inner sep=2pt, minimum size=3mm] (3) at (-1,0) {3};
\node[draw, circle, inner sep=2pt, minimum size=3mm] (2) at (-2,-1) {2};
\node[draw, circle, inner sep=2pt, minimum size=3mm] (1) at (-2,1) {1};

\node at (1,-1.2)  {$4$};
\node at (1,1.2)   {$4$};
\node at (-2.2,0)    {$\infty$};
\node at (-1,1.25)   {$\sqrt{7/3}$};
\node at (-1,-1.25)   {$\sqrt{7/3}$};
\node at (-1.05,0.55)   {$\sqrt{7/3}$};
\node at (-1.1,-0.55)   {$\sqrt{7/3}$};
\node at (0.4,0.2)   {$\sqrt{2}$};

\draw (1) -- (2) node[above,midway]   {};
\draw[dashed] (2) -- (3) node[above,midway] {};
\draw[dashed] (1) -- (3) node[above,midway]   {};
\draw[dashed] (1) -- (4) node[above,midway]   {};
\draw[dashed] (2) -- (5) node[above,midway]   {};
\draw (4) -- (5) node[above,midway]   {};
\draw (4) -- (7) node[above,midway]   {};
\draw (5) -- (8) node[above,midway]   {};
\draw[dashed] (3) -- (6) node[above,midway]   {};
\draw (6) -- (7) node[above,midway]   {};
\draw (6) -- (8) node[above,midway]   {};
\draw (7) -- (8) node[above,midway]   {};
\end{tikzpicture}
\end{center}
The diagram $D_2$ is obviously connected, and we can check that its Gram matrix has signature $(4,1,3)$ with non-positive off-diagonal entries; hence, by Vinberg's theorem~\cite[Theorem~2.1]{V}, it defines a hyperbolic Coxeter $4$-polytope $V$. Using CoxIter~\cite{coxiter,coxiter2}, we find that $V$ has finite volume and is non-arithmetic. 
Moreover, like the polytope $P$,
it satisfies \ref{item:a} and \ref{item:b}: 
there is exactly one ideal vertex, corresponding to the unique maximal affine subdiagram of $D_2$~\cite{V}, which is spanned by $1,2,6,7,8$. Its link is a Euclidean right prism over an equilateral triangle (Figure~\ref{fig:prismK}, obtained by using Proposition \ref{prop:comb-iso}, and colored with the rule of Section \ref{sec:visualizing}). 

\begin{figure}
    \centering
    \includegraphics{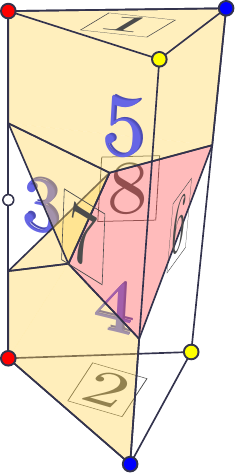}
    \color{black}
    \caption{The projection of $V$ onto the link $K$ of its ideal vertex, with labels indicating the five non-compact facets ($\fac{1},\fac{2},\fac{6},\fac{7},\fac{8}$) and three compact facets ($\fac{3},\fac{4},\fac{5}$). 
    }
    \label{fig:prismK}
\end{figure}

From the diagram $D_2$, we can see that the polytope $V$ has a symmetry $\sigma$ that exchanges the facets $\fac{1}$ and $\fac{2}$, $\fac{4}$ and $\fac{5}$, $\fac{7}$ and $\fac{8}$. This induces a half-turn rotation
of $K$ swapping each vertex with the other one of the same color in Figure~\ref{fig:prismK}. This rotation is the only nontrivial color-preserving combinatorial automorphism of $K$. 

\subsection{Marked tessellations of \texorpdfstring{$3$}{3}-manifolds}
Mirroring the previous sections, we will prove that if a flat $3$-manifold $N$ admits a tessellation in prisms over equilateral triangles with some properties, then there is a cusp-transitive hyperbolic $4$-manifold with cusp type $N$.

\begin{defn}[Marked tessellation]
Let $T$ be a tessellation of a flat $3$-manifold $N$ into right prisms over equilateral triangles, such that the vertices are colored red, yellow or blue. We say that $T$ is \emph{marked} if the vertices of each prism are colored as in Figure~\ref{fig:prismK} (up to combinatorial isomorphism), and whenever two prisms share a facet $F$, they are symmetrical with respect to $F$, including their vertex colorings.
\end{defn}

\begin{rem}
    Given a marked tessellation $T$ and two natural numbers $a,b$, we can \emph{subdivide} each prism as follows. First, note that an equilateral triangle can be subdivided into $a^2$ triangles (whose sides are $a$ times smaller), by cutting it with three sets of equally spaced lines parallel to the sides; similarly, a prism can be subdivided into $a^2$ thin prisms. Each of those can then be cut, parallel to the bases, into $b$ equal prisms. The new tessellation $T'$ can be marked provided that $b$ is odd and that $a-1$ is a multiple of $3$ (see Figure~\ref{fig:subdiv-marked} for the case $a=4$, $b=3$). Indeed, for each prism $t$ of $T$, we now have $b+1$ layers of vertices of $T'$. It is not hard to see that, whenever $a \equiv 1 \pmod{3}$, there is a unique way to $3$-color the bottom layer consistently with the lower three vertices of $t$. Each subsequent layer will be the same as the one below it, but with yellow and blue swapped. If $b$ is odd, then the top and bottom layers are the same with yellow and blue swapped, ensuring consistency with the coloring of $t$.
\end{rem}

\begin{figure}
    \centering
    \includegraphics{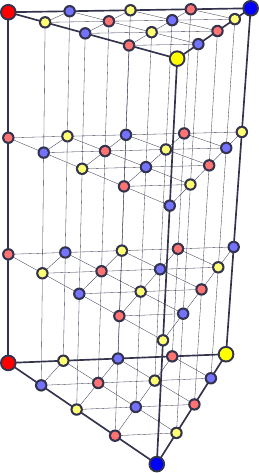}
    \color{black}
    \caption{Subdividing a marked tessellation ($a=4$, $b=3$). The pattern on the bottom face of the prism works whenever $a-1$ is a multiple of $3$.}
    \label{fig:subdiv-marked}
\end{figure}

\begin{rem}\label{rem:facetsV}
When gluing copies of $V$ along non-compact facets, the compact facets merge into two kinds of new facets, which we can visualize with the help of the projection in Figure~\ref{fig:prismK}. The first kind arises from $6$ copies of the facet $\fac{3}$ around the white dot, while the second comes from $12$ copies of the facet $\fac{4}$ or $\fac{5}$ around the corresponding yellow dot. Both kinds of facets are bounded.
\end{rem}

\begin{thm}\label{thm:main2}
    Let $N$ be a closed $3$-manifold that admits a marked tessellation. Then there exists a non-arithmetic cusp-transitive hyperbolic $4$-manifold with cusp type $N$.
\end{thm}
\begin{proof}
    The proof is similar in many aspects to that of Theorem~\ref{thm:main}, so we will only go over the main points.
    Again, by Proposition~\ref{prop:rizzi}, it suffices to prove that there exists a $1$-cusped developable reflectofold with cusp type $N$, obtained by gluing copies of $V$. 

    A marked tessellation of $N$ induces a way to glue copies of $V$ along their non-compact facets; a key fact is that every color-preserving combinatorial automorphism of $K$ is induced by a symmetry of $V$.
    This gives a $1$-cusped reflectofold with cusp type $N$. 

    In order to ensure~\ref{EF} and \ref{AC}, it suffices that each facet of the reflectofold, together with its adjacent facets, is inside an embedded ball. Indeed, this ensures that every facet is embedded and that the intersection of two adjacent facets is a single corner. Since the facets are bounded by Remark~\ref{rem:facetsV}, this can be done by an argument involving the injectivity radius and a sufficiently fine subdivision of the marked tessellation; the reflectofold constructed from the subdivided tessellation is developable.

    The $4$-manifold thus constructed is non-arithmetic since it covers the non-arithmetic orbifold $V/\langle \sigma \rangle$, where $\sigma$ is the order-$2$ isometry of $V$ defined in Section~\ref{sec:poly-V}.
\end{proof}

\begin{cor}
    Let $N$ be one of the manifolds $E_3,E_5$. Then there exists a cusp-transitive hyperbolic $4$-manifold with cusp type $N$.
\end{cor}
\begin{proof}
    The manifolds $E_3$ and $E_5$ have marked tessellations, as shown in Figure~\ref{fig:N2}; the result follows from Theorem~\ref{thm:main2}.
\end{proof}

\begin{figure}
    \centering
    \addtolength{\tabcolsep}{1em}
    \begin{tabular}{cc}
    \includegraphics{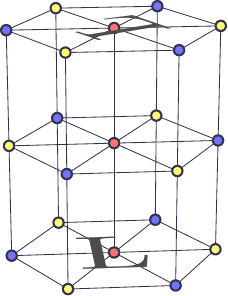}
&   \includegraphics{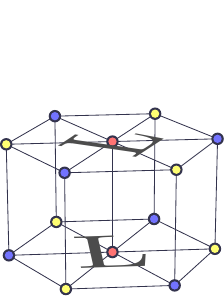}
\\
    \color{black} $\frac 13$-twist & $\frac 16$-twist
    \end{tabular}
    \color{black}
    \addtolength{\tabcolsep}{-1em}
    \caption{Marked tessellations for $E_3$ and $E_5$. Labeled facets are glued according to their labels, while unlabeled facets are glued to the opposite facets by translation.}
    \label{fig:N2}
\end{figure}

\section{Density}\label{sec:density}
In this section we will strengthen the previous results by investigating the possible Euclidean structures that can be realized by the cusp sections of a cusp-transitive $4$-manifold, and we show how the same techniques can be used in the $3$-dimensional case. We have the following result:

\begin{thm}\label{thm:density}
     For every closed flat $3$-manifold $N$, the set of flat metrics on $N$ which can be realized as cusp sections of a cusp-transitive $4$-manifold is dense in the space of all flat metrics of $N$.
\end{thm}
\begin{proof}
The argument is inspired by the proof of~\cite[Theorem~2]{N}.
We divide the proof into two cases: the manifolds $E_1,E_2,E_4,E_6,B_1,B_2,B_3,B_4$ (tessellated by cubes) and the manifolds $E_3,E_5$ (tessellated by triangular prisms).

We start with the first case. Any flat metric on $N$ is induced by a subgroup $\Gamma < \Isom(\matR^3)$. Generators for $\Gamma$ are found in~\cite[Table~1]{N} in the form $(A, t_u)$, where $A \in \mathrm{O}(3)$ and $t_u$ is a translation by the vector $u$, denoting the map $v \mapsto u+Av$.

The group $\Gamma$ can be conjugated as in~\cite[pp.~128--129]{N} so that all the matrices $A_i$ are certain fixed signed permutation matrices, which preserve the $x$ axis, and preserve or exchange the $y$ and $z$ axes. In this \emph{normalized form}, the translation vectors have some zero entries, while the $k$-uple of the other \emph{free} entries can take any value in an open set of $\matR^k$, containing $(\matR\setminus \{0\})^k$.
A dense subset of these forms can be obtained by considering only vectors of the form $v_i\coloneqq (\sqrt{2}a_i, \sqrt{3}b_i, \sqrt{3}c_i)$, where 
$a_i,b_i,c_i$ are in $\matQ\setminus\{0\}$ or $\{0\}$, depending on whether the corresponding entry of $v_i$ is free or not.
Let $\Gamma'$ be a group with parameters in this dense subset; we shall prove that the metric resulting from $\Gamma'$ can be realized by a layered tessellation by copies of $C$ (see Remark~\ref{rem:comb-lay}).

The group $\Gamma'$ preserves a lattice of the form
\[
    \left\{\frac{1}{d}(\sqrt{2}m, \sqrt{3}n, \sqrt{3}p) \mid m,n,p \in \matZ \right\},
\]
where $d$ is a common denominator of all the $a_i,b_i,c_i$. Furthermore, it preserves a layered tessellation of $\matR^3$ by copies of $C$ (having side lengths $\sqrt{2}/d$ and $\sqrt{3}/d$) with the special facets orthogonal to the $x$ axis; this is because the $x$ axis is preserved by the $A_i$. This tessellation descends to the quotient $\matR^3 / \Gamma'$. The result follows as in Theorem~\ref{thm:main}.

As for the second case, we refer to~\cite[Section~4]{conway-rossetti}. The space of flat metrics on $N$ has two parameters: in the hexagonal prism fundamental domains of Figure~\ref{fig:N2}, they can be taken as the side length and height of the prisms. When subdividing a marked tessellation in the proof of Theorem~\ref{thm:main2}, we can choose two parameters $a$ and $b$ provided that $b$ is odd, $a-1$ is a multiple of $3$, and they are sufficiently large. The effect of these parameters is to multiply the side length and height of the fundamental domains by $a$ and $b$ respectively. Since we can realize the tessellation with arbitrarily small copies of $K$, by choosing $a$ and $b$ in an appropriate way, we can approximate any values of the two parameters of the metric of $N$. The result follows as in Theorem~\ref{thm:main2}.
\end{proof}

\subsection{Dimension \texorpdfstring{$3$}{3}}
We consider the following diagram $D_3$, found in~\cite{chein}:

\begin{center}
\begin{tikzpicture}[thick,scale=1.1,  every node/.style={scale=0.8}]
\node[draw, circle, inner sep=2pt, minimum size=3mm] (1) at (-1,0)  {1};
\node[draw, circle, inner sep=2pt, minimum size=3mm] (2) at (0,0)  {2};
\node[draw, circle, inner sep=2pt, minimum size=3mm] (3) at (0.866,0.5)   {3};
\node[draw, circle, inner sep=2pt, minimum size=3mm] (4) at (0.866,-0.5)   {4};

\node at (-0.5,0.2)  {$4$};

\draw (1) -- (2) node[above,midway]   {};
\draw (2) -- (3) node[above,midway] {};
\draw (2) -- (4) node[above,midway]   {};
\draw (3) -- (4) node[above,midway]   {};
\end{tikzpicture}
\end{center}

The corresponding Coxeter polytope $Y$ is a finite-volume hyperbolic tetrahedron with one ideal vertex, corresponding to the subdiagram spanned by $2$, $3$, $4$; its link is an equilateral triangle. The polytope has the properties~\ref{item:a} and~\ref{item:b}.

\begin{thm}\label{thm:density-3d}
     For every closed flat $2$-manifold $S$ (i.e.~the torus and the Klein bottle), the set of flat metrics on $S$ which can be realized as cusp sections of a cusp-transitive $3$-manifold is dense in the space of all flat metrics of $S$.
\end{thm}
\begin{proof}
    We begin by gluing six copies of $Y$ around the edge between the facets $\fac{3}$ and $\fac{4}$. The resulting polytope $Z$ has one ideal vertex, whose link is a regular hexagon, and one compact facet, which meets the other six facets at an angle of $\pi/4$.

    If $S$ has a tessellation by regular hexagons, such that every hexagon is embedded, then we can glue copies of $Z$ along their non-compact facets in the same pattern. The resulting manifold with corners $R$ is a developable reflectofold: its dihedral angles are all equal to $2\cdot \pi/4 = \pi/2$, also implying \ref{AC}, and its facets are embedded hyperbolic hexagons~\ref{EF}.
    Moreover, the section of its unique cusp is isometric to $S$ up to global rescaling. As usual, by Proposition~\ref{prop:rizzi}, we can construct a cusp-transitive manifold which covers $R$, with cusps isometric to $S$.

    It remains to show that, up to arbitrarily small perturbations, every flat metric on the torus or the Klein bottle admits a tessellation by regular hexagons.
    
    First, consider a parallelogram fundamental domain $D_T$ for the torus, and overlay it onto a tessellation of small regular hexagons. We can perturb $D_T$ by moving three vertices to the nearest hexagon center, and the fourth in such a way as to make a parallelogram; the fourth vertex will also fall into the center of a hexagon. This gives our desired tessellation.

    As for the Klein bottle, the generic fundamental domain $D_K$ is a rectangle, with two opposite sides $s_1$,$s_2$ identified with a twist. We overlay $D_K$ onto a tessellation of small regular hexagons with some sides parallel to $s_1$,$s_2$. We can perturb $D_K$ by translation or scaling along either axis, obtaining a rectangle with all vertices at the centers of hexagons. Note that the tessellation is symmetrical with respect to the line joining the midpoints of the perturbed $s_1$ and $s_2$. Hence, we have a tessellation of the Klein bottle.
\end{proof}

\section{Lots of manifolds}\label{sec:lots}
In this section we prove the following theorem.
\begin{thm}\label{thm:lots}
    For every closed flat $3$-manifold $N$, there exists a positive constant $c$ such that, for sufficiently large $V>0$, there exist at least $V^{cV}$ complete hyperbolic $4$-manifolds with pairwise isometric cusps of type $N$ and volume $\le V$. 
\end{thm}
\begin{proof}
    Let $R$ be a reflectofold constructed as in the proof of Theorem \ref{thm:main} or \ref{thm:main2} (according to $N$). Let $D$ be the Coxeter diagram with one vertex for each facet of $R$, where if two facets meet with dihedral angle of $\pi/k$ (resp.~do not intersect), the corresponding vertices are joined by an edge labeled $k$ (resp.~a dashed edge).

    If $R$ is constructed from a sufficiently fine subdivision of a tessellation, then the diagram $D$ is connected. Indeed, two non-adjacent facets are connected by a dashed edge, while for any two adjacent facets $F$, $F'$ there exists a third one not adjacent to them (in the projection onto the link, it can be found in a large embedded ball centered on, say, $F$). As a consequence, we can also assume that $D$ has a dashed edge.
    
    Let $G$ be the Coxeter group associated to $D$. It is neither affine nor spherical, since $D$ has a dashed edge (compare~\cite[Tables~1-2]{V}). By~\cite[Corollary~2]{large}, $G$ has a finite-index subgroup $H$ with a quotient isomorphic to the free group $F_2$. Since $H$ is a subgroup of a Coxeter group, it has a torsion-free subgroup $H'$ such that $d\coloneqq [G:H'] < +\infty$. The image of $H'$ in $F_2$ is free of rank at least $2$, so $H'$ also has a quotient isomorphic to $F_2$. Recall that $F_2$ has at least $r\cdot r!$ subgroups of index $\le r$~\cite{hall}. Hence, by pulling back to $H'$, we obtain at least $r\cdot r!$ torsion-free subgroups of $G$ of index $\le dr$. These correspond to manifold covers of $R$ with degree $\le dr$, which have pairwise isometric cusps of type $N$ by construction. As in the proof of Proposition~\ref{prop:rizzi}, we can show that these manifolds are complete.

    Let $v$ be the volume of $R$ and let $V \coloneqq vdr$. Then our manifolds have volume $\le V$. Using Stirling's approximation, we have the following estimate (for some $k,k',c > 0$ and for $r$ large):

    \begin{align*}
        \log (r\cdot r!)
        &= \log r + \log r!
    \\  &= \log r + r \log r - r + O(\log r) 
    \\  &\ge k r\log r 
    \\  &= k' \frac{V}{vd}\log \frac{V}{vd}
    \\  &\ge c V\log V
    \\  &= \log(V^{cV}).
    \end{align*}

    Hence, we have at least $V^{cV}$ torsion-free subgroups of $G$. The associated manifolds (of volume $\le V$) are not necessarily distinct; however, the same estimate holds on the number of isometry classes, with a smaller constant $c$. 
    Indeed, if $G$ is non-arithmetic, we conclude with the same argument of \cite[``The lower bound'']{MR1952926} using Margulis' theorem on the commensurator \cite[Theorem~1, p.~2]{MR1090825}, while if $G$ is arithmetic, we conclude as in \cite[Section $5.2$]{MR2726109} using the Kazhdan--Margulis theorem.
\end{proof}


\section*{References}
\printbibliography[heading=none]

@book {MR1090825,
    AUTHOR = {Margulis, G. A.},
     TITLE = {Discrete subgroups of semisimple {L}ie groups},
    SERIES = {Ergebnisse der Mathematik und ihrer Grenzgebiete (3) [Results
              in Mathematics and Related Areas (3)]},
    VOLUME = {17},
 PUBLISHER = {Springer-Verlag, Berlin},
      YEAR = {1991},
     PAGES = {x+388},
      ISBN = {3-540-12179-X},
   MRCLASS = {22E40 (20Hxx 22-02 22D40)},
  MRNUMBER = {1090825},
MRREVIEWER = {Gopal\ Prasad},
       DOI = {10.1007/978-3-642-51445-6},
       URL = {https://doi.org/10.1007/978-3-642-51445-6},
}

@article{large,
    AUTHOR = {Margulis, G. A. and Vinberg, É. B.},
     TITLE = {Some linear groups virtually having a free quotient},
   JOURNAL = {J. Lie Theory},
  FJOURNAL = {Journal of Lie Theory},
    VOLUME = {10},
      YEAR = {2000},
    NUMBER = {1},
     PAGES = {171--180},
      ISSN = {0949-5932},
   MRCLASS = {22E40 (20F55)},
  MRNUMBER = {1748082},
MRREVIEWER = {Alexander\ Lubotzky},
}

@article {MR1952926,
    AUTHOR = {Burger, M. and Gelander, T. and Lubotzky, A. and Mozes, S.},
     TITLE = {Counting hyperbolic manifolds},
   JOURNAL = {Geom. Funct. Anal.},
  FJOURNAL = {Geometric and Functional Analysis},
    VOLUME = {12},
      YEAR = {2002},
    NUMBER = {6},
     PAGES = {1161--1173},
      ISSN = {1016-443X,1420-8970},
   MRCLASS = {22E40 (57N16)},
  MRNUMBER = {1952926},
MRREVIEWER = {Igor\ Rivin},
       DOI = {10.1007/s00039-002-1161-1},
       URL = {https://doi.org/10.1007/s00039-002-1161-1},
}

@article {MR2726109,
    AUTHOR = {Belolipetsky, M. and Gelander, T. and Lubotzky,
              A. and Shalev, A.},
     TITLE = {Counting arithmetic lattices and surfaces},
   JOURNAL = {Ann. of Math. (2)},
  FJOURNAL = {Annals of Mathematics. Second Series},
    VOLUME = {172},
      YEAR = {2010},
    NUMBER = {3},
     PAGES = {2197--2221},
      ISSN = {0003-486X,1939-8980},
   MRCLASS = {11N45 (22E40)},
  MRNUMBER = {2726109},
MRREVIEWER = {B.\ Sury},
       DOI = {10.4007/annals.2010.172.2197},
       URL = {https://doi.org/10.4007/annals.2010.172.2197},
}

@book{D,
  title={The Geometry and Topology of Coxeter Groups.},
  author={Davis, M. W.},
  year={2012},
  publisher={Princeton University Press},
  note = {London Mathematical Society Monographs, volume 32},
}

@article{P,
	author = {Prokhorov, M. N.},
	title ={Absence of discrete groups of reflections with a noncompact fundamental polyhedron of finite volume in a {L}obachevskiĭ space of high dimension},
	journal = {Izv. Akad. Nauk SSSR Ser. Mat.},
	volume = {50},
	number = {2},
	year = {1986},
	pages = {413-424},
}

@article{FKS,
	author = {Ferrari, L. and Kolpakov, A. and Slavich, L.},
	title = {Cusps of hyperbolic 4-manifolds and rational homology spheres },
	journal = {Proc. Lond. Math. Soc.},
	volume = {6},
	year = {2021},
	number = {3},
	pages = {636--648}
}

@article{N,
	author = {Nimershiem, B. E.},
	title = {All flat three-manifolds appear as cusps of hyperbolic four-manifolds},
	journal = {Topology and its Appl.},
	volume = {90},
	year = {1998},
	pages = {109-133},
}

@article {LR,
	author = {Long, D. D. and Reid, A. W.},
	title = {On the geometric boundaries of hyperbolic 4-manifolds},
	journal = {Geom. Topol.},
	volume = {4},
	number = {1},
	year = {2000},
	pages = {171-178},
}

@article{V,
	author = "Vinberg, {\`E}. B.",
	fjournal = "Akademiya Nauk SSSR i Moskovskoe Matematicheskoe Obshchestvo. Uspekhi Matematicheskikh Nauk",
	journal = "Uspekhi Mat. Nauk",
	number = "1(241)",
	pages = "29--66, 255",
	title = "{Hyperbolic reflection groups}",
	volume = "40",
	year = "1985"
}

@article {KS,
author = {Kolpakov, A. and Slavich, L.},
title = {Hyperbolic 4-manifolds, colourings and mutations},
journal = {Proc. Lond. Math. Soc.},
volume = {113},
number = {2},
pages = {163--184},
year = {2016},
}

@article {KM,
    AUTHOR = {Kolpakov, A. and Martelli, B.},
     TITLE = {Hyperbolic four-manifolds with one cusp},
   JOURNAL = {Geom. Funct. Anal.},
  FJOURNAL = {Geometric and Functional Analysis},
    VOLUME = {23},
      YEAR = {2013},
    NUMBER = {6},
     PAGES = {1903--1933},
}

@article {IH,
    AUTHOR = {{Im Hof}, H.-C.},
    label  = {IH},
     TITLE = {Napier cycles and hyperbolic {C}oxeter groups},
   JOURNAL = {Bull. Soc. Math. Belg. S\'{e}r. A},
    VOLUME = {42},
      YEAR = {1990},
    NUMBER = {3},
     PAGES = {523--545},
}

@misc{conway-rossetti,
      title={Describing the platycosms}, 
      author={J. H. Conway and J. P. Rossetti},
      year={2003},
      eprint={math/0311476},
      archivePrefix={arXiv},
      primaryClass={math.DG}
}

@misc{rizzi,
      title={Some cusp-transitive hyperbolic 4-manifolds}, 
      author={E. Rizzi},
      year={2024},
      eprint={2402.13209},
      archivePrefix={arXiv},
      primaryClass={math.GT}
}

@book{complete,
    title={Foundations of Hyperbolic Manifolds},
    author={J. G. Ratcliffe},
    year={2019},
    publisher={Springer Cham},
}

@article{B1-mfd,
    author = {Kolpakov, A. and Slavich, L.},
    title = "{Symmetries of Hyperbolic 4-Manifolds}",
    journal = {International Mathematics Research Notices},
    volume = {2016},
    number = {9},
    pages = {2677-2716},
    year = {2015},
    month = {07},
    issn = {1073-7928},
    doi = {10.1093/imrn/rnv210},
    url = {https://doi.org/10.1093/imrn/rnv210},
    %eprint = {https://academic.oup.com/imrn/article-pdf/2016/9/2677/7375306/rnv210.pdf},
}

@article{B2-mfd,
author = {J. G. Ratcliffe and S. T. Tschantz},
title = {Hyperbolic 24-Cell 4-Manifolds With One Cusp},
journal = {Experimental Mathematics},
volume = {32},
number = {2},
pages = {269--279},
year = {2023},
publisher = {Taylor \& Francis},
doi = {10.1080/10586458.2021.1926010},
URL = {https://doi.org/10.1080/10586458.2021.1926010},
}

@software{coxiter2,
  author = {R. Guglielmetti},
  title = {CoxIter},
  note = {\url{https://rgugliel.github.io/CoxIter/index.html}},
  version = {1.3},
}

@article{coxiter,
    title={CoxIter – Computing invariants of hyperbolic Coxeter groups},
    volume={18},
    DOI={10.1112/S1461157015000273},
    number={1},
    journal={LMS Journal of Computation and Mathematics},
    author={Guglielmetti, R.},
    year={2015}, 
    pages={754-–773}
}

@article{hall,
    title={Subgroups of Finite Index in Free Groups},
    volume={1},
    DOI={10.4153/CJM-1949-017-2},
    number={2},
    journal={Canadian Journal of Mathematics},
    author={Hall, M.},
    year={1949},
    pages={187–190}
}

@article{chein,
author = {M. Chein},
journal = {ESAIM: Mathematical Modelling and Numerical Analysis - Modélisation Mathématique et Analyse Numérique},
language = {fre},
number = {R3},
pages = {3-16},
publisher = {Dunod},
title = {Recherche des graphes des matrices de Coxeter hyperboliques d’ordre $\leqslant 10$},
url = {http://eudml.org/doc/193126},
volume = {3},
year = {1969},
}

@online{felikson,
author = {A. Felikson and P. Tumarkin},
title = {Hyperbolic Coxeter polytopes},
note = {\url{https://www.maths.dur.ac.uk/users/anna.felikson/Polytopes/polytopes.html}}
}

@article{davis-lectures,
  title={Lectures on orbifolds and reflection groups},
  author={Davis, M. W.},
  journal={Transformation groups and moduli spaces of curves},
  volume={16},
  pages={63--93},
  year={2011},
  publisher={International Press Somerville}
}

\end{document}